\documentclass[11pt]{amsart}

\usepackage{amsmath}
\usepackage{amssymb}
\usepackage{bm}
\usepackage{graphicx}
\usepackage{psfrag}
\usepackage{color}
\usepackage{hyperref}
\hypersetup{colorlinks=true, linkcolor=blue, citecolor=magenta, urlcolor=wine}
\usepackage{url}
\usepackage{algpseudocode}
\usepackage{fancyhdr}
\usepackage{mathtools}
\usepackage{tikz-cd}
\usepackage{xy}
\input xy
\xyoption{all}
\usepackage{stmaryrd}
\usepackage{calrsfs}

\newtheorem*{maintheorem*}{Main Theorem}
\newtheorem{theorem}{Theorem}[section]
\newtheorem*{theorem*}{Main Theorem}

\newtheorem{question}[theorem]{Question}
\newtheorem{prop}[theorem]{Proposition}

\newtheorem{lemma}[theorem]{Lemma}
\newtheorem{cor}[theorem]{Corollary}

\theoremstyle{definition}
\newtheorem{definition}[theorem]{Definition}

\newtheorem{example}[theorem]{Example}
\numberwithin{equation}{section}

\newcommand{\cc}{\mathbb{C}}

\newcommand{\nn}{\mathbb{N}}
\newcommand{\pp}{\mathbb{P}}
\newcommand{\qq}{\mathbb{Q}}
\newcommand{\rr}{\mathbb{R}}
\newcommand{\zz}{\mathbb{Z}}

\newcommand{\uu}{\mathcal{U}}

\providecommand\ldb{\llbracket}
\providecommand\rdb{\rrbracket}

\newcommand{\gp}{\text{gp}}

\newcommand{\supp}{\textsf{supp}}



\begin{document}
	
\mbox{}
\title{Factorization in additive monoids of evaluation polynomial semirings}

\author{Khalid Ajran}
\address{Department of Mathematics\\MIT\\Cambridge, MA 02139, US}
\email{kajran@mit.edu}

\author{Juliet Bringas}
\address{Matcom, Universidad de La Habana, Habana 10400, Cuba}
\email{julybm01@gmail.com}

\author{Bangzheng Li}
\address{Department of Mathematics\\MIT\\Cambridge, MA 02139, US}
\email{liben@mit.edu}

\author{Easton Singer}
\address{Department of Mathematics\\Harvard University\\Cambridge, MA 02138, US}
\email{esinger@college.harvard.edu}

\author{Marcos Tirador}
\address{Matcom, Universidad de La Habana, Habana 10400, Cuba}
\email{marcosmath44@gmail.com}

\date{\today}
	
\begin{abstract}
	For a positive real $\alpha$, we can consider the additive submonoid $M$ of the real line that is generated by the nonnegative powers of $\alpha$. When $\alpha$ is transcendental, $M$ is a unique factorization monoid. However, when~$\alpha$ is algebraic, $M$ may not be atomic, and even when $M$ is atomic, it may contain elements having more than one factorization (i.e., decomposition as a sum of irreducibles). The main purpose of this paper is to study the phenomenon of multiple factorizations inside $M$. When~$\alpha$ is algebraic but not rational, the arithmetic of factorizations in $M$ is highly interesting and complex. In order to arrive to that conclusion, we investigate various factorization invariants of $M$, including the sets of lengths, sets of Betti elements, and catenary degrees. Our investigation gives continuity to recent studies carried out by Chapman, et al. in 2020 and by Correa-Morris and Gotti in~2022.
\end{abstract}

\bigskip
\maketitle

\section{Introduction}
\label{sec:intro}

Following Gotti \cite{fG19}, we say that an additive submonoid of $\rr$ consisting of nonnegative numbers is a positive monoid. Let $M$ be a positive monoid. A nonzero element of $M$ is called an atom (or an irreducible) if it cannot be written as a sum of two nonzero elements of $M$, and $M$ is called atomic if every nonzero element is the sum of finitely many atoms. For $r \in M$, a formal sum of atoms of $M$ whose actual sum in $M$ is $r$ is called a factorization of~$r$ in~$M$. Factorizations in positive monoids have been studied in~\cite{fG19} by Gotti, in~\cite{mBA19,mBA20} by Bras-Amoros, and more recently in~\cite{GV23} by Gotti and Vulakh. In this paper, we will study factorizations in a special class of positive monoids: for each positive real $\alpha$, the smallest positive monoid $\nn_0[\alpha]$ containing all the nonnegative powers of $\alpha$.
\smallskip

A positive monoid that is closed under multiplication is called a positive semiring. Factorizations in the multiplicative structure of positive semirings were studied by Cesarz et al. in~\cite{CCMS09}, where the elasticity of two classes of positive monoids were considered (the elasticity is an arithmetic invariant that somehow measures how chaotic the phenomenon of multiple factorizations is in a given monoid). Then Baeth and Gotti in~\cite{BG20} used examples of positive semirings to construct multiplicative monoids of matrices with certain desired factorization properties. Motivated by this work, the atomicity of both the additive and the multiplicative structures of further classes of positive monoids was studied by Baeth, Chapman, and Gotti in~\cite{BCG21}. Examples of positive semirings exhibiting a variety of factorization properties can also be found in the recent papers \cite{GP22,GP23} by Gotti and Polo.
\smallskip

The monoids $\nn_0[\alpha]$ we are interested in this paper are also positive semirings. Indeed, if $\nn_0[X]$ denotes the semiring consisting of all the polynomials with nonnegative integer coefficients, then $\nn_0[\alpha] = \{p(\alpha) \mid p(X) \in \nn_0[X]\}$, which justifies the notation. Accordingly, we call $\nn_0[\alpha]$ the evaluation polynomial semiring at $\alpha$ or just an evaluation polynomial semiring. For a positive rational $q$, factorizations in (the additive monoid of) $\nn_0[q]$ were first studied by Chapman, Gotti, and Gotti in~\cite{CGG20}. A more systematic approach to factorizations of $\nn_0[\alpha]$ in the more general context provided by any positive real $\alpha$ was given by Correa-Morris and Gotti in~\cite{CG22}. The present paper is motivated by~\cite{CG22}, and can be considered a continuation of the latter paper. Two more papers extending the work in~\cite{CGG20,CG22} are~\cite{sZ22} by Zhu and~\cite{JLZ23} by Jiang, Li, and Zhu.
\smallskip

The number of atoms (counting repetitions) in a factorization $z$ of a nonzero $b \in M$ is called the length of $z$, and the set consisting of the lengths of all factorizations of $b$ is called the set of lengths of $b$. After introducing some preliminary notation, definitions, and known results in Section~\ref{sec:prelim}, we move to Section~\ref{sec:sets of lengths}, where we consider the sets of lengths of evaluation polynomial semirings. Sets of lengths are crucial objects in factorization theory (see the survey~\cite{aG16} for more details). Following Baeth and Cassity~\cite{BC12}, we say that an atomic positive monoid is $k$-furcus if every nonzero element has a factorization of length at most~$k$. The property of being $2$-furcus has been considered by Adams et al.~\cite{AABHKMPR11} in the context of matrix monoids, by Geroldinger, Kainrath, and Reinhart~\cite[Section~3]{GKR15} in the context of finitely primary monoids, and by Gotti and O'Neill~\cite[Section~6]{GO20} in the context of Puiseux monoids. The more general property of being $k$-furcus has been considered by Geroldinger and Hassler~\cite[Section~4]{GH08} in the context of $v$-Noetherian monoids. We show that there are no $k$-furcus evaluation polynomial semirings for any positive integer~$k$. 
\smallskip

In Section~\ref{sec:sets of lengths}, we also investigate the structure of sets of lengths of evaluation polynomial semirings. The structure of the sets of lengths of many classes of atomic monoids has been the subject of a great deal of investigation for almost three decades: see \cite[Chapter~4.7]{GH06} by Geroldinger and Halter-Koch for background on the Structure Theorem for Sets of Lengths, and see the recent paper \cite{GK22} by Geroldinger and Khadam (and the references therein) for current results in this direction. For any positive rational~$q$, it was proved by Chapman et al.~\cite{CGG20} that the set of lengths of every nonzero element of $\nn_0[q]$ is an arithmetic progression, provided that $\nn_0[q]$ is atomic. In Subsection~\ref{subsec:sets of lengths}, we extend this result to more general evaluation polynomial semirings. We conclude Section~\ref{sec:sets of lengths} by constructing an evaluation polynomial semiring having sets of lengths that are as different from an arithmetic progression as one could possibly expect.
\smallskip

In Section~\ref{sec:Betti elements}, we study the set of Betti elements of evaluation polynomial semirings. The Betti graph of a nonzero element $b \in M$ is the graph whose vertices are the factorizations of $b$ with an edge between two factorizations precisely when they have an atom in common. Then an element of $M$ is called a Betti element if its Betti graph is disconnected. First, we prove that when $\alpha$ is a positive algebraic number such that $\nn_0[\alpha]$ is atomic, then $\nn_0[\alpha]$ has finitely many Betti elements if and only if it is finitely generated. Then we explicitly compute the set of Betti elements of evaluation polynomial semirings at those positive real numbers~$\alpha$ satisfying that $\alpha^n \in \nn$ for some $n \in \nn$.
\smallskip

In Section~\ref{sec:catenary degree}, we study the catenary degree of evaluation polynomial semirings. The catenary degree, which we will define in the next section, is an arithmetic invariant used to measure how drastic the phenomenon of multiple factorizations is in a given atomic monoid. The catenary degree of atomic positive monoids $\nn_0[q]$, where $q$ is a positive rational, is finite and it was explicitly computed in~\cite[Corollary~3.4]{CGG20}. Although establishing a formula for the catenary degree of any evaluation polynomial semiring seems to be more subtle, at the end of Section~\ref{sec:catenary degree} we present a realization result for the catenary degree of evaluation polynomial semirings. We prove that, for any value $c \in \nn_{\ge 3} \cup \{\infty\}$, there exists a positive algebraic number $\alpha$ (resp., a algebraic number $\beta$) such that $\nn_0[\alpha]$ satisfies the ascending chain condition on principal ideals or ACCP for short (resp., $\nn_0[\beta]$ is atomic but does not satisfy ACCP) and has catenary degree~$c$. The notion of ACCP will be introduced in the next section.

\bigskip
\section{Preliminary}
\label{sec:prelim}

\smallskip
\subsection{General Notation}

As is customary, $\zz$, $\zz/ \hspace{-1pt} m\zz$, $\qq$, $\rr$, and $\cc$ will denote the set of integers, integers modulo $m$, rational numbers, real numbers, and complex numbers, respectively. We let $\nn$ and $\nn_0$ denote the set of positive and nonnegative integers, respectively. In addition, we let $\pp$ denote the set of primes. For $a,b \in \zz$ with $a \le b$, we let $\ldb a,b \rdb$ denote the set of integers between~$a$ and $b$, i.e., $\ldb a,b \rdb = \{n \in \zz \mid a \le n \le b\}$. In addition, for $S \subseteq \rr$ and $r \in \rr$, we set $S_{\ge r} = \{s \in S \mid s \ge r\}$ and $S_{> r} = \{s \in S \mid s > r\}$. For a positive rational $q = \frac nd$, where $n,d \in \nn$ and $\gcd(n,d) = 1$, we let $n$ and $d$ be denoted by $\mathsf{n}(q)$ and $\mathsf{d}(q)$.

\smallskip
\subsection{Atomic Monoids}

Although a monoid is usually defined to be a semigroup with an identity element, here we will additionally assume that all monoids are cancellative and commutative. Let $M$ be a monoid written additively. We set $M^\bullet := M \setminus \{0\}$. The \emph{difference group} $\gp(M)$ of a monoid $M$ is the set of formal differences of elements in~$M$ (i.e., the unique abelian group $\gp(M)$ up to isomorphism satisfying that any abelian group containing a homomorphic image of $M$ will also contain a homomorphic image of $\gp(M)$). The rank of $M$ is the dimension of $\qq \otimes \gp(M)$ as a $\qq$-vector space. The group of invertible elements of~$M$ is denoted by $\uu(M)$. The monoid $M$ is called \emph{reduced} if the only invertible element of $M$ is $0$. An element $a \in M \! \setminus \! \uu(M)$ is an \emph{atom} (or \emph{irreducible}) if whenever $a = u+v$ for some $u,v \in M$, then either $u \in \uu(M)$ or $v \in \uu(M)$. The set of atoms of $M$ is denoted by $\mathcal{A}(M)$. Following Cohn~\cite{pC68}, we say that $M$ is \emph{atomic} if every non-invertible element can be written as a sum of atoms. Let $S$ be a subset of $M$. We let $\langle S \rangle$ denote the smallest submonoid of $M$ containing $S$, and we call $\langle S \rangle$ the submonoid of $M$ \emph{generated} by $S$. If $M = \langle S \rangle$, then $S$ is called a \emph{generating set} of $M$, and $M$ is called \emph{finitely generated} provided that $M$ has a finite generating set.
\smallskip

A subset $I$ of $M$ is an \emph{ideal} of~$M$ provided that $I + M := \{b+c \mid b \in I \text{ and } c \in M\} = I$ or, equivalently, $I + M \subseteq I$. The ideal $I$ is \emph{principal} if the equality $I = b + M$ holds for some $b \in M$. The monoid $M$ is said to satisfy the \emph{ascending chain condition on principal ideals} (or ACCP for short) provided that every ascending chain of principal ideals of $M$ eventually stabilizes. It is well known that every monoid that satisfies the ACCP is atomic (see \cite[Proposition~1.1.4]{GH06}). The converse does not hold in general, and several examples and classes of atomic monoids that do not satisfy the ACCP can be found in the recent paper~\cite{GL23}. The first example of an integral domain whose multiplicative monoid is atomic but does not satisfy the ACCP was given by Grams in~\cite{aG74}.

\smallskip
\subsection{Factorizations and their Lengths}

Observe that the monoid $M$ is atomic if and only if its quotient monoid $M_{\text{red}} = M/\uu(M)$ is atomic. Let $\mathsf{Z}(M)$ be the free (commutative) monoid on the set of atoms $\mathcal{A}(M_{\text{red}})$. The elements of $\mathsf{Z}(M)$ are called \emph{factorizations}. It turns out that the set $\mathsf{Z}(M)$ under the formal sum of atoms is also an atomic monoid. We can define the notion of a greatest common divisor in the monoid $\mathsf{Z}(M)$ as follows: for any two factorizations $z = \sum_{a \in \mathcal{A}(M)} \mu_a a$ and $z' = \sum_{a \in \mathcal{A}(M)} \nu_a a$ in $\mathsf{Z}(M)$ (here all but finitely many coefficients are zero) set
\[
	\gcd(z,z') := \sum_{a \in \mathcal{A}(M)} \min\{\mu_a, \nu_a\} a,
\]
and call $\gcd(z,z')$, which is also an element of $\mathsf{Z}(M)$, the \emph{greatest common divisor} of $z$ and~$z'$. Let $\pi \colon \mathsf{Z}(M) \to M_\text{red}$ be the unique monoid homomorphism fixing the set $\mathcal{A}(M_{\text{red}})$. For any $b \in M$, we set $\mathsf{Z}(b) := \pi^{-1}(b)$; that is, we let $\mathsf{Z}(b)$ be the set of all formal sums $a_1 + \dots + a_\ell$ in $\mathsf{Z}(M)$, where $a_1, \dots, a_\ell \in \mathcal{A}(M_{\text{red}})$, whose actual sum in $M_{\text{red}}$ is $b + \mathcal{U}(M)$. The elements of $\mathsf{Z}(b)$ are called \emph{factorizations} of $b$. If $|\mathsf{Z}(b)| = 1$ for every $b \in M$, then $M$ is called a \emph{unique factorization monoid} (or a UFM for short). On the other hand, if $M$ is atomic and $|\mathsf{Z}(b)| < \infty$ for every $b \in M$, then $M$ is called a \emph{finite factorization monoid} (or an FFM for short). It follows from the definitions that every UFM is an FFM. Every  finitely generated monoid is an FFM (see \cite[Proposition~2.7.8]{GH06}).
\smallskip

If $z := a_1 + \dots + a_\ell$ is a factorization in $\mathsf{Z}(M)$ for $a_1, \dots, a_\ell \in \mathcal{A}(M_{\text{red}})$, then $\ell$ is denoted by $|z|$ and called the \emph{length} of~$z$. Now, for every $b \in M$, we set
\[
	\mathsf{L}(b) = \{ |z| \mid z \in \mathsf{Z}(b) \},
\]
and call $\mathsf{L}(b)$ the \emph{set of lengths} of $b$. If $M$ is atomic and $|\mathsf{L}(b)| < \infty$ for every $b \in M$, then $M$ is called a \emph{bounded factorization monoid} (or a BFM for short). It follows directly from the definitions that every FFM is a BFM. In addition, one can readily argue that every BFM must satisfy the ACCP (\cite[Corollary~1.4.4]{GH06}). The bounded and finite factorization properties were introduced by Anderson, Anderson, and Zafrullah~\cite{AAZ90} in the context of integral domains (see \cite{AG22} for a recent survey), and the same properties were generalized by Halter-Koch~\cite{fHK92} to context of monoids. Examples of positive monoids that are BFMs but do not satisfy the ACCP as well as examples of positive monoids that are FFMs but not BFMs can be found in~\cite{fG22}, and the same type of examples in the setting of positive semirings can be found in~\cite{BCG21}.

\smallskip
\subsection{Betti Elements and the Catenary Degree}

Let $M$ be an atomic monoid. The \emph{Betti graph} $\nabla_b$ of an element $b \in M$ is the graph whose set of vertices is $\mathsf{Z}(b)$ having an edge between $z,z' \in \mathsf{Z}(b)$ if and only if $\gcd(z,z') \neq 0$. An element of $M$ is called a \emph{Betti element} if its Betti graph is disconnected.

We can use the notion of greatest common divisors in $\mathsf{Z}(M)$ to measure the distance between factorizations and to turn $\mathsf{Z}(M)$ into a metric space. We call
\[
	\mathsf{d}(z,z') := \max \big\{ |z| - |\gcd(z,z')|, |z'| - |\gcd(z,z')| \big\}
\]
the distance between $z$ to $z'$ in $\mathsf{Z}(M)$. It is not hard to argue that $\mathsf{Z}(M)$ is a metric space with respect to the distance $\mathsf{d}$ we have just defined. For $N \in \mathbb{N}_0 \cup \{\infty\}$, a finite sequence $z_0, z_1, \dots, z_k$ in $\mathsf{Z}(x)$ is called an $N$-chain of factorizations connecting $z$ and $z'$ if $z_0 = z$, $z_k = z'$, and $\mathsf{d}(z_{i-1}, z_i) \le N$ for every $i \in \ldb 1, k \rdb$. The \emph{catenary degree} of an element $b \in M$, denoted by $\mathsf{c}(b)$, is the smallest $n \in \mathbb{N}_0 \cup \{\infty\}$ such that for any two factorizations in $\mathsf{Z}(b)$, there exists an $n$-chain of factorizations connecting them. The \emph{catenary degree} of $M$ is
\[
	\mathsf{c}(M) := \sup \{ \mathsf{c}(b) \mid b \in M \} \in \mathbb{N}_0 \cup \{ \infty \}.
\]

\smallskip
\subsection{Minimal Pairs and Evaluation Polynomial Semirings}

Let $\alpha \in \mathbb{C}$ be an algebraic number (over $\qq)$. The \emph{support} of $f(X)$ is the set of exponents of the monomials appearing in the canonical representation of a polynomial $f(X) \in \qq[X]$ , and we denote it by $\supp \, f(X)$:
\[
	\supp \, f(X) := \{n \in \nn_0 \mid f^{(n)}(0) \neq 0 \},
\]
where $f^{(n)}$ denotes the $n$-th derivative of $f$. Observe that there exists a unique $\ell \in \nn$ such that the polynomial $\ell m_\alpha(X) \in \zz[X]$ has content~$1$; that is, the greatest common divisor of all the coefficients of $\ell m_\alpha(X)$ is~$1$. Also, there are unique polynomials $p_\alpha(X), q_\alpha(X) \in \nn_0[X]$ such that $\ell m_\alpha(X) = p_\alpha(X) - q_\alpha(X)$ and $\supp \, p_\alpha(X) \bigcap \supp \, q_\alpha(X)$ is empty. We call $(p_\alpha(X), q_\alpha(X))$ the \emph{minimal pair} of $\alpha$.
\smallskip

For any $\alpha \in \rr$, we set
\[
	\nn_0[\alpha] = \{p(\alpha) \mid p(X) \in \nn_0[X]\}.
\]
As mentioned in the introduction, $\nn_0[\alpha]$ is a semiring, which we call the \emph{evaluation polynomial semiring} at $\alpha$ or just an \emph{evaluation polynomial semiring}. The rest of this paper is devoted to study the phenomenon of non-unique (additive) factorizations in the setting of evaluation polynomial semirings. It follows from \cite[Section~5]{GG18} that if $q$ is a positive rational, then $\nn_0[q]$ is atomic if and only if $q^{-1} \notin \nn_{\ge 2}$, in which case $\mathcal{A}(\nn_0[q]) = \{q^n \mid n \in \nn_0\}$. A deeper and more general study of the atomicity of $\nn_0[\alpha]$ when $\alpha$ is a positive algebraic number was recently carried out in~\cite{CG22}.

\bigskip
\section{Sets of Lengths}
\label{sec:sets of lengths}

In this section we address the following two fundamental questions about the sets of lengths of the monoids $\nn_0[\alpha]$ for positive algebraic $\alpha$.

\begin{question}
	Let $\alpha$ be a positive algebraic number.
	\begin{enumerate}
		\item Is there a $k \in \nn$ such that each $x \in \nn_0[\alpha]$ has a factorization of length at most $k$?
		\smallskip
		
		\item Is there an $N \in \nn$ such that the set of lengths of each $x \in \nn_0[\alpha]$ is an arithmetic progression after we delete from it the intervals $\ldb \min \mathsf{L}(x), \min \mathsf{L}(x) + N \rdb$ and $\ldb \sup \mathsf{L}(x) - N, \sup \mathsf{L}(x) \rdb$, where the last interval is the empty set when $\sup \mathsf{L}(x) = \infty$?
	\end{enumerate}
\end{question}

As we will see through this section, these two questions are relevant in the factorization theory literature and have been considered in various classes of monoids with the aim of understanding their arithmetic of factorizations.

\medskip
\subsection{$k$-Furcusness}
\label{subsec:furcuness}

To formalize the first question given in this section, we introduce the notion of $k$-furcusness.

\begin{definition}
	For $k \in \mathbb{N}$, we say that the atomic monoid $M$ is \emph{$k$-furcus} provided that for each non-invertible element $x \in M$ the set $\mathsf{L}(x) \cap \{ 1, \ldots, k \}$ is nonempty.
\end{definition}

Thus, an atomic monoid is $k$-furcus for some $k \in \nn$ provided that every (non-invertible) element of $M$ has a factorization whose length is at most $k$. As our next proposition indicates, no monoid of the form $\nn_0[\alpha]$ is $k$-furcus. First, we need the following known lemma, which we prove here for the sake of completeness. For any two subsets $A$ and $B$ of $\rr$, we set
\[
	A+B := \{a+b \mid a \in A \text{ and } b \in B\}.
\]

\begin{lemma} \label{lem:sum of closed is closed}
	Let $A$ and $B$ be two subsets of $\rr_{\ge 0}$. If $A$ and $B$ are both closed under the Euclidean topology, then so is $A + B$.
\end{lemma}

\begin{proof}
	Suppose, towards a contradiction, that there is a real number $r \notin A + B$ such that $r$ is a limit point of $A + B$. Then there exist infinite sequences $(a_n)_{n \in \nn}$ and $(b_n)_{n \in \nn}$, whose underlying sets are contained in $A$ and $B$, respectively, such that the sequence $(a_n + b_n)_{n \in \nn}$ converges to $r$. Therefore $(a_n + b_n)_{n \in \nn}$ is bounded by some $N \in \nn$, and so the fact that $(b_n)_{n \in \nn}$ is a sequence of positive terms ensures that $(a_n)_{n \in \nn}$ is bounded by $N$. Hence, after replacing $(a_n + b_n)_{n \in \nn}$ by a suitable subsequence, we can assume that $(a_n)_{n \in \nn}$ converges, namely, $\lim_{n \to \infty} a_n = a$ for some $a \in \rr_{\ge 0}$. This, in turn, implies that the sequence $(b_n)_{n \in \nn}$ converges to $r - a$. Since both $A$ and $B$ are closed, $a \in A$ and $r - a \in B$, which implies that $r = a + (r - a) \in A + B$, a contradiction.
\end{proof}

\begin{prop}
	For any $k \in \nn$ and $\alpha \in \mathbb{R}_{> 0}$, the monoid $\mathbb{N}_0[\alpha]$ is not $k$-furcus.
\end{prop}

\begin{proof}
	Suppose, by way of contradiction, that there exist $k \in \nn$ and $\alpha \in \rr_{>0}$ such that the monoid $M := \mathbb{N}_0[\alpha]$ is $k$-furcus. Since $M$ is $k$-furcus but not a group, it cannot be a UFM, and so $\alpha \neq 1$. Now consider the set
	\[
		A := \{0\} \cup \{\alpha^n \mid n \in \mathbb{Z}\}.
	\]
	Observe that $A$ is a countable subset of $\rr_{\ge 0}$ that is closed with respect to the Euclidean topology. Therefore the $k$-fold set $k \cdot A := A + A + \cdots + A$ ($k$ times) is also countable and, in light of Lemma~\ref{lem:sum of closed is closed}, closed under the Euclidean topology. This shows that it cannot be dense in $\mathbb{R}_{\ge 0}$ (as otherwise because it is closed, it would equal $\mathbb{R}_{\ge 0}$, which is not countable). Thus, there exist $a,b \in \rr_{> 0}$ with $a < b$ such that the interval $(a, b)$ is disjoint from $k \cdot A$. Now take $m \in \mathbb{Z}$ such that $(\alpha^m a, \alpha^m b) \cap \nn$ is a nonempty set, and then take $\beta \in (\alpha^m a, \alpha^m b)\cap \mathbb{N}$. Since $\nn \subseteq M^\bullet$ and $M$ is $k$-furcus, we see that $\beta \in M$ and also that $\beta$ has a factorization in $M$ of length $\ell$ with $\ell \le k$. Then after writing $\beta = \sum_{i=1}^\ell \alpha^{m_i}$ for some $m_1, \dots, m_\ell \in \nn_0$, we observe that $\beta/\alpha^m = \sum_{i=1}^\ell \alpha^{m_i - m} \in k \cdot A$. However, $\beta/\alpha^m \in (a,b)$, which contradicts the fact that the interval $(a,b)$ is disjoint from the set $k \cdot A$.
\end{proof}

Observe that, for each $k \in \nn$, any $k$-furcus monoid is $(k+1)$-furcus. On the other hand, for each $k \in \nn$, there are $(k+1)$-furcus monoids that are not $k$-furcus. Here is an additive version of an example of a $(k+1)$-furcus monoid that is not $k$-furcus, provided by Baeth and Cassity in~\cite{BC12}.

\begin{example}
	Fix $k \in \nn$ with $k \ge 3$. Now, for each $i \in \ldb 1,k \rdb$, consider the subset $A_i \subset \nn^k$ defined as follows
	\[
		A_i = \{ (v_1, \dots, v_k) \in \nn^k \mid v_j= 1 \text{ for all } j \in \ldb 1,k \rdb \setminus \{i\} \, \}.
	\]
	Set $A := \bigcup_{i=1}^k A_i $ and consider the additive submonoid $M := \langle A \rangle$ of $\nn_0^k$.  One can easily see that every $k$-tuple in $A$ is an atom of $M$ as such a $k$-tuple has at least one coordinate equal to $1$. Thus, $M$ is atomic with $\mathcal{A}(M) = A$. We shall show that~$M$ is $k$-furcus but not $(k-1)$-furcus.
	
	Assume, by way of contradiction, that $M$ is $(k-1)$-furcus. Now consider the element $u := (k, \dots, k) \in M$ and write $u = a_1 + \dots + a_\ell$ for some $\ell \in \ldb 1, k - 1 \rdb$ and $a_1, \dots, a_\ell \in A$. By virtue of the pigeonhole principle, there exists $j \in \ldb 1, k \rdb$ such that $a_i \notin A_j \setminus \{ (1, 1, \ldots, 1) \}$ for any $i \in \ldb 1, \ell \rdb$. Then every $a_i$ has a $1$ in its $j$-th coordinate and, therefore, the $j$-th coordinate of $u$ equals $\ell$, which is a contradiction. Hence $M$ is not $(k-1)$-furcus.
	
	Finally, let us prove that $M$ is $k$-furcus. To do so, consider any element $v := (v_1, \dots, v_k)$ having a factorization of length larger than $k$. Observe that $v_i > k$ for every $i \in \ldb 1, k \rdb$. For each $i \in \ldb 1, k \rdb$, set $a_i := (1, \dots, v_i - k + 1, \dots, 1)$, that is, $a_i$ is the $k$-tuple with $1$ in every coordinate except for the $i$-th coordinate, which is equal to $v_i - k + 1$. Since $v = \sum_{i=1}^k a_i$, we see that $v$ has a factorization of length $k$. Hence $M$ is $k$-furcus.
\end{example}

\medskip
\subsection{Structure of Sets of Lengths}
\label{subsec:sets of lengths}

It was proved in \cite[Theorem~3.3]{CGG20} that for any $q \in \qq_{> 0}$ such that $\nn_0[q]$ is atomic, the set of lengths of every element of $\nn_0[q]$ is a (perhaps finite) arithmetic progression. In this subsection we will generalize this result to monoids $\nn_0[\alpha]$, where $\alpha$ is the positive $n$-th root of a positive rational.
\smallskip

For $q \in \qq_{> 0}$, define an \emph{irreducible root} of $q$ to be an $n$-th root of $q$, namely, $\alpha = \sqrt[n]{q}$ such that $\alpha, \alpha^2, \dots, \alpha^{n-1}$ are all irrational. The following lemma is a known result.

\begin{lemma} \cite[page 297]{sL02} \label{lem:linearly independency}
	Let $q$ be a positive rational, and let $n \ge 2$ such that $\sqrt[n]{q}$ is an irreducible $n$-th root of $q$. Then the polynomial $x^n - q$ is irreducible over $\qq$.
\end{lemma}

In order to establish the next proposition, we need the following consequence of Lemma~\ref{lem:linearly independency}.

\begin{cor} \label{cor:linearly indepencency}
	Let $q$ be a positive rational, and let $\alpha := \sqrt[n]{q}$ be a positive irreducible $n$-th root of $q$. Then $1,\alpha, \dots, \alpha^{n-1}$ are linearly independent over $\qq$.
\end{cor}

\begin{proof}
	It is a consequence of Lemma~\ref{lem:linearly independency}: indeed, if $1,\alpha, \dots, \alpha^{n-1}$ were linearly dependent over $\qq$, then the polynomial $x^n - q$ (whose unique positive real root is $\alpha$) would be reducible over $\qq$.
\end{proof}

In order to understand the atomicity and arithmetic of $\nn_0[\alpha]$ when $\alpha$ is a positive irreducible root of a positive rational, the following lemma will play a crucial role.

\begin{lemma} \label{lem:factorizations in quasi-rational cyclic semiring}
	Let $q$ be a non-integer positive rational that is not the reciprocal of any integer, and let $\alpha := \sqrt[n]{q}$ be the positive irreducible $n$-th root of $q$. Let $f \colon \mathsf{Z}(\nn_0[q])^n \to \mathsf{Z}(\nn_0[\alpha])$ be the map defined by $f(z_1, z_2, \ldots, z_n) = z_1 + \alpha z_2 + \alpha^2 z_3 + \cdots + \alpha^{n-1} z_n$. Then the following statements hold.
	\begin{enumerate}
		\item $f$ is a monoid isomorphism.
		\smallskip
		
		\item If $\pi(f(z_1, z_2, \dots, z_n)) = \pi(f(z_1' , z_2' , \dots, z_n' ))$, then $\pi(z_i) = \pi(z_i')$ for every $i \in \ldb 1, n \rdb$.
	\end{enumerate}
\end{lemma}

\begin{proof}
	(1) One can readily check that $f$ is additive and, therefore, it is a monoid homomorphism. To show that it is an isomorphism, we will show that every factorization in $\mathsf{Z}(\nn_0[\alpha])$ corresponds to a unique $n$-tuple of factorizations in $\mathsf{Z}(\nn_0[q])^n$, which argues both injectivity and surjectivity. Since $\nn_0[\alpha]$ is generated by the set $\{\alpha^k \mid k \in \nn_0\}$, its set of atoms consists of nonnegative powers of $\alpha$. Moreover, as $\alpha^n = q$, any atom $\alpha^{k'}$ can be written in the form $\alpha^r q^k$, where $k' = kn + r$ with $k \in \nn_0$ and $r \in \ldb 0, n-1 \rdb$. Thus, any factorization $z \in \mathsf{Z}(\nn_0[\alpha])$ can be written as
	\begin{equation} \label{eq:factorization lemma}
		z = \sum_{r=0}^{n-1} \sum_{k=0}^\infty c_{r,k} \alpha^r q^k =  \sum_{r=0}^{n-1} \alpha^r \bigg( \sum_{k=0}^\infty c_{r,k} q^k \bigg)
	\end{equation}
	for some coefficients $c_{r,k} \in \nn_0$. It follows from \cite[Section~5]{GG18} that the set of atoms of $\nn_0[q]$ is $\{ q^n \mid n \in \nn_0 \}$, from which one can deduce from~\eqref{eq:factorization lemma} that the unique $n$-tuple of factorizations in $\mathsf{Z}(\nn_0[q])^n$ that is sent to $z$ by $f$ is $(\sum_{k=0}^\infty c_{0,k} q^k, \sum_{k=0}^\infty c_{1,k} q^k, \dots, \sum_{k=0}^\infty c_{n-1,k} q^k)$. Thus, $f$ is a monoid isomorphism.
	\smallskip
	
	(2) The equality $\pi(f(z_1, z_2, \dots, z_n)) = \pi(f(z_1' ,z_2' , \dots, z_n' ))$ implies that
	\[
		\pi(z_1) + \alpha \pi(z_2) + \cdots + \alpha^{n-1} \pi(z_n) = \pi(z_1') + \alpha \pi(z_2')+ \cdots + \alpha^{n-1} \pi(z_n'),
	\]
	and so the linear independence in Corollary~\ref{cor:linearly indepencency} guarantees that $\pi(z_i) = \pi(z_i')$ for every $i \in \ldb 1, n \rdb$.
\end{proof}

\begin{prop} \label{prop:atomicity and lengths of quasi-rational cyclic semirings}
	Let $q$ be a non-integer positive rational that is not the reciprocal of any integer, and let $\alpha := \sqrt[n]{q}$ be a positive irreducible $n$-th root of $q$. Then the following statements hold.
	\begin{enumerate}
		\item $\nn_0[\alpha]$ is atomic with set of atoms $\mathcal{A}(\nn_0[\alpha]) = \{\alpha^k \mid k \in \nn_0\}$.
		\smallskip
		
		\item For any $b \in \nn_0[\alpha]^\bullet$, the set $\mathsf{L}(b)$ is an arithmetic progression with difference $|\mathsf{n}(q) - \mathsf{d}(q)|$.
	\end{enumerate}
\end{prop}

\begin{proof}
	(1) Any element that is generated by at least two generating elements cannot be an atom, so $\mathcal{A}(\nn_0[\alpha]) \subseteq \{\alpha^k \mid k \in \nn_0\}$. Now fix $k' \in \nn_0$, and let us verify that $\alpha^{k'} \in \mathcal{A}(\nn_0[\alpha])$. To do so, write $k' = k n + r$ for some $k \in \nn_0$ and $r \in \ldb 0, n - 1 \rdb$. Hence $\alpha^{k'} = \alpha^r q^k$. Since the isomorphism $f$ of Lemma~\ref{lem:factorizations in quasi-rational cyclic semiring} takes atoms to atoms, the set of factorizations  of $\alpha^{k'}$ in $\nn_0[\alpha]$ is in bijection with the set of factorizations of $v := (0, 0, \ldots, 0, q^k, 0, \ldots, 0)$ in $\nn_0[q]^n$, where the only nonzero entry $q^k$ of $v$ occupies the $r$-th position. Since $q^k$ is an atom of $\nn_0[q]$, it follows that $\alpha^{k'}$ is an atom of $\nn_0[\alpha]$. Thus, $\mathcal{A}(\nn_0[\alpha]) = \{ \alpha^k \mid k \in \nn_0\}$.
	\smallskip
	
	(2) Let $b = t_0 + t_1 \alpha + t_2 \alpha^2 + \dots +t_{n-1} \alpha^{n-1}$, where $t_0, \dots, t_{n-1} \in \qq_{\ge 0}$. It follows from Lemma~\ref{lem:factorizations in quasi-rational cyclic semiring} that the set of factorizations of $b$  is in bijection with the set of factorizations of $T := (t_0, t_1, \dots, t_{n-1})$ in $\nn_0[q]^n$. The length of a certain factorization $(z_0, z_1, \dots, z_{n-1})$ of $T$ equals $\sum_{i=0}^{n-1} |z_i|$. Set $d := |\mathsf{d}(q) - \mathsf{n}(q)|$. It follows from \cite[Theorem~3.3]{CGG20} that, for each $i \in \ldb 0, n-1 \rdb$, the length of the factorization $z_i$ of $t_i$ can take any value in the set
	\[
		\bigg\{ \min \mathsf{L}(t_i) + dk \ \bigg{|} \ k \in \ldb 0, \frac{\max \mathsf{L}(t_i) - \min \mathsf{L}(t_i)}{d} \rdb \bigg\}.
	\] 
	Moreover, since the terms of the $n$-tuple are independent, $\min \mathsf{L}(T) = \sum_{i=0}^{n-1} \min \mathsf{L}(t_i)$. Thus, $|(z_0, z_1, \dots, z_{n-1})|$ can take any value in
	\[
		\bigg\{ \min \mathsf{L}(T) + dk \ \bigg{|} \ k \in \ldb 0, \frac{\max \mathsf{L}(T) - \min \mathsf{L}(T)}{d} \rdb \bigg\}.
	\]
	Hence, $\mathsf{L}(T)$ is an arithmetic progression with difference $d$ and it follows from Lemma~\ref{lem:factorizations in quasi-rational cyclic semiring} 
	that the same holds true for $\mathsf{L}(b)$.
\end{proof}

In general, this is not the case for $\nn_0[\alpha]$ when $\alpha$ is not rational. In the following example, we describe an atomic monoid $\nn_0[\alpha]$ for some positive algebraic~$\alpha$ containing an element whose set of lengths is finite but not an arithmetic progression.

\begin{example} \label{ex:a finite set of length not an arithmetic progression}
	Consider the polynomial $m(X) = X^4 - 6X^3 + 4X^2 - 2X - 2$. From Eisenstein's Criterion we know that $m(X)$ is irreducible in $\zz[X]$. As $m(0) = -2$ we deduce that $m$ has a positive root $\alpha$. Now consider the monoid $M := \nn_0[\alpha]$. Observe that
	\[
		m(X)(X^2 + X + 1) = X^6 - 5X^4 - X^4 - 4X^3 - 4X - 2,
	\]
	from which one can readily deduce that $M$ is an FGM. Consider for every $b \in \zz$ the polynomial $g_b(X) = m(X)(X + b)$ and observe that the coefficients $ -6b +4 $ and $4b - 2$ (corresponding to degrees $3$ and $2$, respectively) cannot be simultaneously negative. Hence we have that $\alpha^5 \notin \langle \alpha^0, \alpha^1, \dots,  \alpha ^ 4 \rangle$ and therefore $\mathcal{A}(M) = \{\alpha^i \mid i \in \ldb 0,5 \rdb \}$.
	
	Now consider the factorization $z := \alpha^5 + 6\alpha^2 \in \textsf{Z}(M)$ and let $x = \pi(z)$. We know that every other factorization of $x$ is obtained from $z + m(\alpha)r(\alpha)$ for some polynomial $r(X) \in \zz[X]$. Indeed, by taking $r(X) = -(X + 1)$ and $r(X) = -(X + 2)$, we see that
	\[
		z_1 = 5\alpha^4 + 2\alpha^3 + 4\alpha^2 + 4\alpha + 2 \quad \text{and} \quad z_2 = 4\alpha^4 + 8\alpha^3 + 6\alpha + 4
	\]
	are factorizations of $x$, respectively. Observe that $|z| = 7$, $|z_1| = 17$ and $|z_2| = 22$. Let us prove that $x$ does not have a factorization of length $12$, this will be enough to conclude that the set of lengths of $x$ is not an arithmetic progression. To do so, let $p(X) = X^5 + 6X^2$ and assume that there exists some $q(X) \in \nn_0[X]$ satisfying that $q(\alpha)$ is a factorization of length $12$ of $x$. Then for some $r \in \zz[X]$, the following equality holds:
	\[
		q(X) = m(X)r(X) + p(X).
	\]
	Taking into account that $q(1) = |q(\alpha)| = 12$, we can evaluate the previous equation at $X = 1$ to obtain that $r(1) = -1$. Hence $r(X) = aX - a - 1$ for some $a \in \zz$. If $a = -1$, then $q^{(3)}(0) = -4$ (where $q^{(n)}(0)$ is the coefficient corresponding to degree $n$ of $q$), which is not possible since all of the coefficients of $q$ belong to $\nn_0$. If $a < -1$, then $q^{(5)}(0) < 0$ and in other case $q^{(4)}(0) < 0$. In any case, we obtain a contradiction, and so $\mathsf{L}(x)$ is not an arithmetic progression. Moreover, since $M$ is an FGM, it is also a FFM, and thus $\mathsf{L}(x)$ is finite.
\end{example}

Let us recall the notion of an almost arithmetic progression. For $d \in \mathbb{N}$ and $N \in \mathbb{N}_0$, we say that $S \subseteq \mathbb{Z}$ is an \emph{almost arithmetic progression} (\emph{AAP}) with difference $d$ and bound $N$ if
\[
	S = c + \big( S' \cup S^* \cup S'' \big) \subseteq c + d \mathbb{Z},
\]
where $c \in \mathbb{Z}$ and $S^*$ is a nonempty arithmetic progression with difference $d$ such that $\min S^* = 0$ while $S' \subseteq \{ -N, \ldots, -1 \}$ and $S'' \subseteq \sup S^* +  \{ 1, \dots, N \}$ (we assume that $S''$ is empty provided that $S^*$ is infinite). Monoids $M$ for which there exist $d \in \nn$ and $N \in \nn_0$ such that every set of length of $M$ is an AAP with difference $d$ and bound $N$ are said to satisfy the \emph{structure property for sets of lengths}. Observe that every arithmetic progression is an AAP. We have seen in part~(2) of Proposition~\ref{prop:atomicity and lengths of quasi-rational cyclic semirings} a class of monoids of the form $\nn_0[\alpha]$ whose members satisfy the structure property for sets of lengths. Under certain reasonable conditions, strongly primary monoids satisfy the structure property for sets of lengths (see~\cite[Theorem~4.3.6]{GH06} for more details).

In sharp contrast with Proposition~\ref{prop:atomicity and lengths of quasi-rational cyclic semirings} and Example~\ref{ex:a finite set of length not an arithmetic progression}, in the next example we exhibit an atomic monoid $\nn_0[\alpha]$ (for a positive algebraic number $\alpha$) which does not satisfy the structure property for sets of lengths.

\begin{example} \label{ex:set of lengths that is not an AAP}

Consider the polynomial $p(X) = X^4 - X^3 - X^2 - X + 1$. One can verify that the polynomial $p(X)$ has precisely two real roots, namely, $\alpha$ and $\beta$ with $\alpha \approx 1.722$ and $\beta \approx 0.581$. We want to prove that we cannot choose $d \in \nn$ and $N \in \nn_0$ such that the set of lengths of each element in $M := \mathbb{N}_0[\alpha]$ is an AAP.

Let us argue first that $\mathcal{A}(M) = \{\alpha^n \mid n \in \mathbb{N}_0\}$. Since $\alpha > 1$, we see that $0$ is not a limit point of $M^\bullet$, which implies that $M$ is a BFM; this is \cite[Theorem~4.5]{fG19} (indeed, $M$ is an FFM via \cite[Theorem~5.6]{fG19} or \cite[Theorem~4.11]{CG22}). The monoid $\mathbb{N}_0[\beta]$ is isomorphic to $M$ by \cite[Lemma~3.1]{CG22} and, therefore, atomic. From $\beta < 1$, we can deduce that $\mathcal{A}(\mathbb{N}_0[\beta]) = \{ \beta^n \mid n \in \mathbb{N}_0 \}$ (see \cite[Theorem~4.1]{CG22}). Thus, $|\mathcal{A}(M)| = |\mathcal{A}(\mathbb{N}_0[\beta])| = \infty$, and so it follows from \cite[Theorem~4.1]{CG22} that $\mathcal{A}(M) = \{\alpha^n \mid n \in \nn_0\}$.

Now suppose, by way of contradiction, that there exist $d \in \nn$ and $N \in \nn_0$ such that for every element $x \in M$ the set $\mathsf{L}(x)$ is an AAP with difference $d$ and bound $N$. Then, for each $x \in M$, the set $\mathsf{L}(x)$ satisfies the following property: if $\lvert \mathsf{L}(x) \rvert \ge 2$, then for any $\ell \in \mathsf{L}(x)$, there exists $\ell' \in \mathsf{L}(x) \setminus \{\ell\}$ such that $\lvert \ell - \ell' \rvert \le \max(d, N)$. For each $k \in \nn$, set $y_k := \alpha^{3k + 1} + 1 \in M$, and let us prove the following claim.
\smallskip

\noindent \textit{Claim.} $y_k = \alpha^{3k} + \alpha + \sum_{i = 1}^k \alpha^{3i - 1}$ for every $k \in \nn$.
\smallskip

\noindent \textit{Proof of Claim.} We use induction on $k$. Observe that when $k = 1$, 
\[
	 \alpha^{3k} + \alpha + \sum_{i = 1}^k \alpha^{3i - 1} = \alpha^3 + \alpha + \alpha^2 = \alpha^4 + 1 = y_1
\]
because $\alpha$ is a root of $p(X)$. Now assume that $k \ge 2$ and that $y_j = \alpha^{3j} + \alpha + \sum_{i = 1}^j \alpha^{3i - 1}$ for every $j \in \ldb 1, k-1 \rdb$. Observe that
\begin{align*}
	y_k & = y_{k - 1} + \alpha^{3k + 1} - \alpha^{3k - 2}  = \alpha^{3k - 3} + \alpha + \alpha^{3k + 1} - \alpha^{3k - 2} + \sum_{i = 1}^{k - 1} \alpha^{3i - 1} \\
	& = (\alpha^{3k + 1} - \alpha^{3k} - \alpha^{3k - 1} - \alpha^{3k - 2} + \alpha^{3k - 3}) + \alpha^{3k} + \alpha + \sum_{i = 1}^k \alpha^{3i - 1} \\
	& = \alpha^{3k} + \alpha + \sum_{i = 1}^k \alpha^{3i - 1}.
\end{align*}
As a result, our claim follows from induction.
\smallskip

Then the equality $\mathcal{A}(M) = \{\alpha^n \mid n \in \nn_0\}$, in tandem with the previous claim, guarantee that, for each $k \in \nn$, the element $y_k$ has factorizations of length $2$ and $k+2$, namely, those determined by the expressions  $\alpha^{3k + 1} + 1$ and $\alpha^{3k} + \alpha + \sum_{i = 1}^k \alpha^{3i - 1}$. Thus, $\{2, k + 2\} \subseteq \mathsf{L}(y_k)$ and, in particular, $\lvert \mathsf{L}(y_k) \rvert \ge 2$. Now set $D := \max(d, N) + 2$. Since $1 \not\in \mathsf{L}(y_k)$, from our assumption we can deduce that for any $k \in \nn$ there exists $\ell_k \in \ldb 3, D \rdb$ with $\ell_k \in \mathsf{L}(y_k)$.

Finally, consider the set $A := \{0\} \cup \{\alpha^n \mid n < 0 \}$. Notice that $A$ is closed in $\mathbb{R}_{\ge 0}$ with respect to the Euclidean topology. Therefore it follows from Lemma~\ref{lem:sum of closed is closed} that the set $D \cdot A := A + A + \cdots + A$ ($D$ times) is also closed under the Euclidean topology. We can write $\alpha^{3k + 1} + 1 = \sum_{i = 1}^{\ell_k} \alpha^{m_i}$ for some $m_1, \dots, m_{\ell_k} \in \ldb 0, 3k \rdb$. As a consequence,
\[
	1 + \alpha^{-3k - 1} = \sum_{i = 1}^{\ell_k} \alpha^{m_i - 3k - 1} \in D \cdot A.
\]
As the sequence $(\alpha^{-(3n+1)})_{n \in \nn}$ converges to zero, we see that $1$ is a limit point of $D \cdot A$ and, as $D \cdot A$ is closed with respect to the Euclidean topology, $1 \in D \cdot A$. However, this contradicts the fact that $\mathcal{A}(M) = \{\alpha^n \mid n \in \mathbb{N}_0\}$.
\end{example}

\bigskip
\section{Sets of Betti Elements}
\label{sec:Betti elements}

In this section we study the sets of Betti elements of evaluation polynomial semirings. It turns out that a monoid of the form $\nn_0[\alpha]$ is finitely generated if and only if it is atomic and has finitely many Betti elements. Let us begin by proving this result.

\begin{prop} \label{prop:fg iff finitely many Bettis}
	Let $\alpha $ be a positive algebraic number such that $\nn_0[\alpha]$ is atomic. Then the following conditions are equivalent.
	\begin{enumerate}
		\item[(a)] $\nn_0[\alpha]$ is finitely generated.
		\smallskip
		
		\item[(b)] $\nn_0[\alpha]$ has finitely many Betti elements.
	\end{enumerate}
\end{prop}

\begin{proof}
	(a) $\Rightarrow$ (b): If $\nn_0[\alpha]$ is finitely generated, then it is finitely presented by Redei's theorem and, therefore, it has finitely many Betti elements.
	\smallskip
	
	(b) $\Rightarrow$ (a): Set $M := \nn_0[\alpha]$. Assume that $M$ is not finitely generated, and let us argue that~$M$ has infinitely many Betti elements. Because $M$ is atomic and non-finitely generated, it follows from \cite[Theorem~4.2]{CG22} that $\mathcal{A}(M) = \{ \alpha^n \mid n \in \nn_0\}$. Let $(p(x),q(x))$ be the minimal pair of $\alpha$. Consider the set
	\[
		L = \{\ell \in \nn \mid \text{there exists }  x \in \nn_0[\alpha] \text{ with } \ell \in \mathsf{L}(x) \text{ and } |\mathsf{Z}(x)| \ge 2 \}.
	\]
	Observe that $L$ is nonempty because $p(\alpha)$ and $q(\alpha)$ determine distinct factorizations of the same element of $M$, and so $p(1) \in L$. Set $m :=  \min L$. Now take $x \in M$ such that $m \in \mathsf{L}(x)$ and $|\mathsf{Z}(x)| \geq 2$, and then take $z \in \mathsf{Z}(x)$ with $|z| = m$. We claim that $z$ is an isolated vertex in $\nabla_x$. Suppose, to the contrary, that this is not the case. Then there exists $z' \in \mathsf{Z}(x)$ with $z' \neq z$ such that $\gcd(z,z')$ consists of at least one atom. Take $y \in M$ such that $w := z - \gcd(z,z')$ and $w' := z' - \gcd(z,z')$ both belong to $\mathsf{Z}(y)$. Since $w \neq w'$, we see that $\vert z-\gcd(z,z')\vert \in L$ and $|z - \gcd(z,z')| < m$, which contradicts the minimality of $m$ in~$L$. Hence $z$ is an isolated vertex in $\nabla_x$, as claimed. Since $|\mathsf{Z}(x)| \ge 2$, this implies that~$x$ is a Betti element of $M$.
	
	Finally, consider the sequence $(x_n)_{n \in \nn}$, where $x_n :=\alpha^n x \in M$ for every $n \in \nn$. Since $\mathcal{A}(M) = \{ \alpha^n \mid n \in \nn_0\}$, for each factorization $z$ of $x$, we see that the $\alpha^n z$ determines a factorization of $x_n$ of length $|z|$. Therefore $m \in \mathsf{L}(x_n)$ and $\mathsf{Z}(x_n) \ge 2$ for all $n \in \nn$. Proceeding as we did in the previous paragraph, we can show that the graph $\nabla_{x_n}$ must contain an isolated vertex. Therefore $x_n$ is a Betti element of $M$ for every $n \in \nn$. Thus, $M$ has infinitely many Betti elements.
\end{proof}

Although the implication (a) $\Rightarrow$ (b) in Proposition~\ref{prop:fg iff finitely many Bettis} holds for any monoid, this is not the case for the implication (b) $\Rightarrow$ (a). The following example illustrates this observation.

\begin{example}
	Consider the Puiseux monoid $M = \langle \frac{1}{p} \mid p \in \pp \rangle$. It is well known and not hard to verify that $M$ is atomic with $\mathcal{A}(M) = \{ \frac{1}{p} \mid p \in \pp \}$ (see \cite{GG18}). In addition, one can  see that an element $q \in M$ has a unique factorization if and only if $1 \nmid_M q$. Take $x \in M$ with $x > 1$ such  that $1 \mid_M x$. Then the element $y := x-1$ belongs to $M$. Fix $a_1, \dots, a_\ell \in \mathcal{A}(M)$ such that $a_1 + \dots + a_\ell \in \mathsf{Z}(y)$. Now, consider two factorizations $z_1, z_2 \in \mathsf{Z}(x)$, and pick $\frac{1}{p_1}, \frac{1}{p_2} \in \mathcal{A}(M)$ such that $\frac{1}{p_1}$ and $\frac{1}{p_2}$ appear in $z_1$ and  $z_2$, respectively. Then
	\[
		z_1' := p_1 \cdot \frac{1}{p_1} + a_1 + \dots + a_\ell \quad \text{ and } \quad z_2' := p_2 \cdot \frac{1}{p_2} + a_1 + \dots + a_\ell
	\]
	both belong to $\mathsf{Z}(x)$, and so $z_1, z_1', z_2', z_2$ is a path in $\nabla_x$. This ensures that $z_1$ and $z_2$ are in the same connected component in $\nabla_x$. As a consequence, the graph $\nabla_x$ is connected. Finally, note that $\mathsf{Z(1)} = \{ p\cdot \frac{1}{p} \mid p \in \mathbb{P}\}$. Then $\nabla_1$ is a graph with infinitely many isolated vertices. Thus, $1$ is the only Betti element of $M$.
\end{example}

Our next goal is to determine the Betti elements of $\nn_0[\alpha]$, where $\alpha$ is a positive irreducible $n$-th root of a positive rational. In the next lemma we tackle the case where $n=1$, and we resolve the general case in the subsequent proposition. Let $\text{Betti}(M)$ denote the set of Betti elements of a given atomic monoid $M$.

\begin{lemma} \label{lem:Betti element in the rational case}
	Let $q$ be a non-integer positive rational such that $\nn_0[q]$ is atomic. Then the set of Betti elements of $\nn_0[q]$ is
	\[
		\emph{Betti}(\nn_0[q]) = \bigg\{ \frac{\mathsf{n}(q)^{m+1}}{\mathsf{d}(q)^m} \ \Big{|} \ m \in \nn_0 \bigg\}.
	\]
\end{lemma}

\begin{proof}
	Set $n := \mathsf{n}(q)$ and $d := \mathsf{d}(q)$. We first argue the following claim.
	\smallskip
	
	\noindent \textit{Claim.} If an element of $\nn_0[q]$ has more than one factorization, then all factorizations of that element have length at least $\min(d, n)$.
	\smallskip
	
	\noindent \textit{Proof of Claim.} Let $p(X)$ and $p'(X)$ be two polynomials in $\nn_0[X]$ such that the formal evaluations $z := p(q)$ and $z' := p'(q)$ yield two distinct factorizations of the number $p(q)$. Then, the polynomial $p(X) - p'(X) \in \mathbb{Z}[X]$ has $q$ as a root, so it can be written as $(dX-n)r(X)$ for some nonzero $r(X) \in \mathbb{Z}[X]$. Assume, without loss of generality, that $p(X) - p'(X)$ has positive leading coefficient, so $r(X)$ does as well. Then, the leading coefficient of $(dX-n)r(X)$ is at least $d$, which implies that the number of copies of that power of $q$ in the factorization $z$ has to be at least $d$. Hence $|z| \geq d$. Furthermore, there must be a smallest nonnegative integer $j$ such that $r^{(j)}(0) > 0$; that is, the coefficient of $X^j$ in $r(X)$ is positive. As a result, the coefficient of $X^j$ in $(dX-n)r(X)$ gets a contribution of at most $0$ from the $dX$ and a contribution of at most $-n$ from the $-n$, so the coefficient of $X^j$ in $(dX-n)r(X)$ is at most $-n$. This implies that $z'$ must contain at least $n$ copies of $q^j$ and, therefore, $|z'| \geq n$. Now the fact that both $d$ and $n$ are at least $\min(d, n)$ ensures that $\min(|z|, |z'|) \geq \min(d, n)$. Finally, since $z$ and $z'$ are arbitrary factorizations of $p(q)$, each factorization of $p(q)$ has length at least $\min(d, n)$, which establishes our claim.
	\smallskip

    We then show that all Betti elements must have the form $\frac{n^{m + 1}}{d^m}$. Suppose, by way of contradiction, that there exists $r \in \nn_0[q]^\bullet$ with $r \neq \frac{n^{m+1}}{d^m}$ for any $m \in \nn_0$ such that $r$ has at least two distinct factorizations $z, z' \in \mathsf{Z}(r)$ in $\nn_0[q]$ that are disconnected in the Betti graph $\nabla_r$, and then take polynomials $p(X), p'(X) \in \nn_0[X]$ whose formal evaluations $p(q)$ and $p'(q)$ yield the factorizations $z$ and $z'$, respectively. Furthermore, we can take $z$ and $z'$ minimizing $t = \max(\deg p(X), \deg p'(X))$ and, among all such factorizations $z$ and $z'$ minimizing $t$, we can further assume that we have chosen $z$ and $z'$ so that the sum of their coefficients corresponding to the term $X^t$ is as small as it could possibly be. Now write $p(X) - p'(X) = (d X - n) r(X)$ where $r \in \zz[X]$. Then we see that the coefficient of the leading term of $p(X) - p'(X)$ is a multiple of $d$. Because $z$ and $z'$ is disconnected in $\nabla_r$, without loss of generality we can assume that $p(X) = d s X^t + p_0(X)$ with $s \in \nn$ and $\max( \deg p_0(X), \deg p'(X)) < t$. If $p_0 = 0$ and $s = 1$, then $r = \frac{n^t}{d^{t - 1}}$, which contradicts our assumption. Therefore $z$ and $z'' := d(s - 1) q^t + n q^{t - 1} + p_0(q)$ are connected in $\nabla_r$ and, by assumption, $z''$ and $z'$ are connected in $\nabla_r$. As a result, $z$ and $z'$ are connected in $\nabla_r$, which is a contradiction.
    \smallskip
	
	Now we will show that each $r = \frac{n^{m + 1}}{d^m}$ is a Betti element by splitting the rest of the proof into two cases.
	\smallskip
	
	\noindent \textit{Case 1:} $d < n$. For each $m \in \nn_0$ we see that $\frac{n^{m+1}}{d^m} \in \mathbb{N}_0[q]$ and also that the same element has at least two factorizations, namely, $n \cdot q^m$ and $d \cdot q^{m+1}$. We claim that $d \cdot q^{m + 1}$ is an isolated vertex in the Betti graph of $\frac{n^{m+1}}{d^m}$.

	Suppose, by way of contradiction, that there was another factorization of $\frac{n^{m+1}}{d^m}$ that had a nonzero greatest common divisor with $d \cdot q^{m+1}$. Observe that this factorization would have at least one copy of the atom $q^{m+1}$, so if we subtracted $q^{m+1}$ from it, we would obtain a second factorization of the number $(d-1) \cdot q^{m+1}$. However, this violates the statement of our claim: we would have an element of $\mathbb{N}_0[q]$ with multiple factorizations that has a factorization whose length is less than $d = \min(d, n)$. Thus, $d \cdot q^{m+1}$ is an isolated vertex in the Betti graph of $\frac{n^{m+1}}{d^m}$ for each $m \in \nn_0$, and so $\big\{ \frac{\mathsf{n}(q)^{m+1}}{\mathsf{d}(q)^m} \ \big{|} \ m \in \nn_0 \big\} \subseteq \text{Betti}(\nn_0[q])$.
 
	\smallskip
	
	\noindent \textit{Case 2:} $n < d$. The argument for this case will be very similar, but a few details are different. For $m \in \nn_0$, observe that $\frac{n^{m+1}}{d^m} \in \mathbb{N}_0[q]$ and also that the same element has at least two factorizations in $\nn_0[q]$, namely, $n \cdot q^m$ and $d \cdot q^{m+1}$. We claim that $n \cdot q^m$ is an isolated vertex in the Betti graph of $\frac{n^{m+1}}{d^m}$.

	Suppose, for the sake of a contradiction, that there was another factorization of $\frac{n^{m+1}}{d^m}$ that had a nonzero greatest common divisor with $n \cdot q^m$. This factorization would have at least one copy of the atom $q^m$, so if we subtracted $q^m$ out, we would obtain a second factorization of the element $(n-1) \cdot q^m$. This violates the statement of our initial claim because an element of $\mathbb{N}_0[q]$ with more than one factorization has a factorization whose length is less than $n = \min(d, n)$. Therefore $n \cdot q^m$ is an isolated vertex in the Betti graph of $\frac{n^{m+1}}{d^m}$ for each $m \in \nn_0$. Thus, $\big\{ \frac{\mathsf{n}(q)^{m+1}}{\mathsf{d}(q)^m} \ \big{|} \ m \in \nn_0 \big\} \subseteq \text{Betti}(\nn_0[q])$.
\end{proof}

We are in a position now to determine the Betti elements of $\nn_0[\alpha]$, where $\alpha$ is a positive irreducible $n$-th root of a positive rational.

\begin{prop}
	Let $q$ be a positive rational that is neither an integer nor the reciprocal of an integer, and let $\alpha := \sqrt[n]{q}$ be a positive irreducible $n$-th root of $q$. Then the set of Betti elements of $\nn_0[\alpha]$ is
	\[
		\bigg\{ \frac{\mathsf{n}(q)^{m+1}}{\mathsf{d}(q)^m} \alpha^r \ \Big{|} \ m \in \nn_0 \text{ and } r \in \ldb 0, n-1 \rdb \bigg\}.
	\]
\end{prop}

\begin{proof}
	For a fixed $r \in \ldb 0, n-1 \rdb$, consider the set $B_r := \{t \alpha^r \in \text{Betti}(\nn_0[\alpha]) \mid t \in \qq_{> 0}\}$. For $t \alpha^r \in B_r$, it follows from Lemma~\ref{lem:factorizations in quasi-rational cyclic semiring} that the set of factorizations of $t \alpha^r$ is in a natural bijection with the set of factorizations of the element $(0,0, \dots, 0, t, 0, \dots, 0)$ in the monoid $\nn_0[q]^n$, where the only nonzero entry $t$ occupies the $(r + 1)$-th position. 
	It follows from Lemma~\ref{lem:Betti element in the rational case} that
	\[
		\text{Betti}(\nn_0[q]) = \bigg\{ \frac{\mathsf{n}(q)^{m+1}}{\mathsf{d}(q)^m} \ \Big{|} \ m \in \nn_0 \bigg\}.
	\]
	This, along with Lemma~\ref{lem:factorizations in quasi-rational cyclic semiring}, guarantees that 
	\[
		B_r = \bigg\{ \frac{\mathsf{n}(q)^{m+1}}{\mathsf{d}(q)^m} \alpha^r \ \Big{|} \ m \in \nn_0 \bigg\}.
	\]
	
	We proceed to argue that all elements of $\nn_0[\alpha]$ that involve at least two different (linearly independent) powers of $\alpha$ cannot be Betti elements. Let $b = t_0 + t_1 \alpha + t_2 \alpha^2 + \cdots + t_{n-1} \alpha^{n-1}$ be such an element, where $t_i \in \qq_{\ge 0}$, and at least two of the $t_i$s are nonzero. It follows from Lemma~\ref{lem:factorizations in quasi-rational cyclic semiring} that the set of factorizations of $b$ is in a natural bijection with the set of factorizations of $(t_0, t_1, \dots, t_{n-1})$ in $\nn_0[q]^n$. Let $z$ and $z'$ be two arbitrary factorizations of $b$, and let $(z_0, z_1, \dots , z_{n-1})$ and $(z_0' , z_1' , \dots, z_{n'-1})$ be the $n$-tuples of factorizations corresponding to them via the map $f$ in Lemma~\ref{lem:factorizations in quasi-rational cyclic semiring}. The equality $\pi(z) = \pi(z') = b$, along with part~(2) of Lemma~\ref{lem:factorizations in quasi-rational cyclic semiring}, ensures that $\pi(z_i) = \pi(z_i')$ for every $i \in \ldb 0, n-1 \rdb$. As a result,
	\[
		(z_0, z_1, \dots, z_{n-1}) \to (z_0', z_1, \dots, z_{n-1}) \to (z_0' ,z_1', \dots, z_{n-1}) \to \dots \to (z_0' ,z_1', \dots, z_{n-1}')
	\]
	is a sequence consisting of factorizations of $(t_0, t_1 , \dots, t_{n-1})$. Moreover, since at least two of these factorizations are nonzero, each two consecutive factorizations have nonzero greatest common divisor, and so this is a path within the factorization graph $\nabla_{(t_0, t_1, \dots, t_{n-1})}$. Therefore $(t_0, t_1, \dots, t_{n-1})$ is not a Betti element and, therefore, it follows from Lemma~\ref{lem:factorizations in quasi-rational cyclic semiring} that $b$ is not a Betti element.
	
	As a consequence, we conclude that
		\[
	 		\text{Betti}(\nn_0[\alpha]) = \bigcup_{r=0}^{n - 1} B_r = \bigg\{ \frac{\mathsf{n}(q)^{m+1}}{\mathsf{d}(q)^m} \alpha^r \ \Big{|} \ r \in \ldb 0, n - 1 \rdb \ \text{ and } \ m \in \nn_0 \bigg\}.
		\]
\end{proof}

\bigskip
\section{The Catenary Degree}
\label{sec:catenary degree}

In this section, we study the catenary degree of evaluation polynomial semirings. The catenary degree of $\mathbb{N}_0[\alpha]$ when $\alpha$ is a positive rational is known. Indeed, by virtue of \cite[Corollary~3.4]{CGG20}, for each $q \in \qq_{> 0}$ such that $\nn_0[q]$ is atomic but not a UFM, the catenary degree of $\nn_0[q]$ equals $\max\{\textsf{n}(q), \mathsf{d}(q)\}$.

\begin{lemma}
	Let $\alpha$ be a positive algebraic number such that the monoid $\nn_0[\alpha]$ is atomic. If $\mathsf{c}( \nn_0[\alpha]) \le 2$, then $\nn_0[\alpha]$ is a UFM.
\end{lemma}

\begin{proof}
	Assume that $\mathsf{c}( \nn_0[\alpha]) \le 2$. It is well known that an atomic monoid whose catenary degree is at most $2$ must be an HFM. Therefore $\nn_0[\alpha]$ is an HFM. Now it follows from \cite[Theorem~5.4]{CG22} that $\nn_0[\alpha]$ is a UFM.
\end{proof}

Atomic monoids of the form $\nn_0[\alpha]$ (for positive algebraic $\alpha$) can be used to realize any possible value of catenary degrees. This is proved in the following theorem.

\begin{theorem}
	For every $c \in \{0\} \cup \nn_{\ge 3}$ and $d \in \nn$, there exists a positive algebraic $\alpha$ such that $\nn_0[\alpha]$ is a rank-$d$ atomic monoid with $\mathsf{c}(\nn_0[\alpha]) = c$. Moreover, for any $c \in \nn_{\ge 3}$ and $d \in \nn$ the following statements hold.
	\begin{enumerate}
		\item There exists a positive algebraic $\alpha$ such that $\nn_0[\alpha]$ has rank $d$, satisfies the ACCP, and $\mathsf{c}(\nn_0[\alpha]) = c$.
		\smallskip
		
		\item If $c \in \nn$, then there exists a positive algebraic $\alpha$ such that $\nn_0[\alpha]$ has rank $d$, is atomic, does not satisfy the ACCP, and $\mathsf{c}(\nn_0[\alpha]) = c$.
	\end{enumerate}
\end{theorem}

\begin{proof}
	Suppose first that $c=0$ and $d \in \nn$. The polynomial $m(X) = X^d - 2$ is irreducible by Eisenstein's Criterion. Let $\alpha$ be the positive root of $m(X)$. It follows from \cite[Proposition~3.2(2)]{CG22} that $\text{rank} \, \nn_0[\alpha] = d$. In addition, $\nn_0[\alpha]$ is a UFM in light of \cite[Theorem~5.4]{CG22} and, therefore, $\mathsf{c}(\nn_0[\alpha]) = 0$. The case when $c \in \nn_{\ge 3}$ in the first statement of the theorem is a special case of part~(1), which we proceed to argue after introducing a useful definition for the rest of our proof.
	
	For every factorization $z = \sum_{i=0}^{k} c_i \alpha^i \in \mathsf{Z}(M)$, we define $s(z)$ as the greatest $t$ such that $c_t \geq c$ or $s(z) = -1$ if $c_t < c$ for every $t \in \ldb 0,k \rdb$.
	\smallskip
	
	(1) Suppose that $c \in \nn_{\ge 3}$ and $d \in \nn$.
	Consider the polynomial $m(X) = pX^d - c$ for some $p \in \pp$ such that $p < c$ and $p\nmid c$. Since $m(1)<0$, $m(X)$ has a positive real root $\alpha$ such that $\alpha > 1$. Let $\overline{m}(X)= cX^d-p$ be the reciprocal polynomial of $m(X)$. By virtue of Eisenstein's Criterion $\overline{m}(X)$ is irreducible over $\zz$. It is known that $m(X)$ must be also irreducible over $\zz$. Then $m(X)$ is the minimal polynomial of $\alpha$. It follows from \cite[Proposition 3.2]{CG22} that $\text{rank }M = \deg m(X)= d$. Since $\alpha > 1$, we see that $0$ is not a limit point of $\nn_0[\alpha]^\bullet$ and, therefore, it follows from \cite[Proposition~4.5]{fG19} that $M$ is a BFM, and so it satisfies the ACCP.  In particular, $M$ is atomic. In addition, it is not hard to verify that $\mathcal{A}(\nn_0[\alpha]) = \{ \alpha^n \mid n \in \nn_0 \}$.
	
	\smallskip
	\noindent \textit{Claim.} For every $z \in \mathsf{Z}(M)$ there exists $\gamma_z \in \mathsf{Z}(\pi(z))$ such that $s(\gamma_z) = -1$ and  we can find a $c$-chain of factorizations connecting $z$ and $\gamma_z$.

	\smallskip
	\noindent \textit{Proof of Claim.}  We use induction on $s(z)$. Observe that when $s(z) = -1$ we can take $\gamma_z = z$ and our claim follows. Assume that $k \in \nn_0$ and that our claim holds for every $z \in \mathsf{Z}(M)$ such that $s(z) < k$.  Now take $z \in \mathsf{Z}(x)$ such that $s(z) = k$ for some $x \in M$. Write $z = \sum_{i=0}^{t} c_i \alpha^i$ where $c_i \in \nn_0$ for every $i \in \ldb 0,t \rdb$. Observe that $p \alpha^{k+d} = c \alpha^{k}$ since $\alpha$ is a root of $X^{k} m(X)$. Let $a$ and $b$ be the unique positive integers satisfying that $c_k = ca + b$ and  $b < c$. Consider the chain $z = z_0, z_1, \dots, z_a$ of factorizations  of $x$ defined as 
	\[		
		z_j = \sum_{i \in \ldb 1, t \rdb \setminus \{ k\} } \! \! \! c_i \alpha^i + (p j) \alpha^{k+d} + (c(a - j) + b) \alpha^{k},
	\]
	for every $j \in \ldb 0 , a\rdb$. Notice that $\mathsf{d}(z_{j-1}, z_{j}) = c$ for each $j \in \ldb 1,a \rdb$. Hence $z_0, z_1, \dots, z_a$ is a $c$-chain of factorizations of $x$. Moreover, $s(z_a) < s(z) = k$, and by induction hypothesis  there exists a $\gamma_z \in \mathsf{Z}(x)$ with $s(\gamma_z) = -1$ and some $c$-chain of factorizations connecting $z_a$ and $\gamma_z$. Our claim follows after concatenating both chains. 
	
	Fix $x \in M$. We can easily check that there exists at most one $\gamma \in \mathsf{Z}(x)$ such that $s(\gamma) = -1$. Then from our claim we have that for any two factorizations $z, w \in \mathsf{Z}(x)$ there exist two $c$-chain of factorizations $z = z_0, z_1, \dots, z_a = \gamma_z$ and $w = w_0, w_1, \dots, w_{a'} = \gamma_w$ such that $s(\gamma_z) = s(\gamma_w) = -1$. Hence $\gamma_z = \gamma_w = \gamma$ and therefore $z = z_0, \dots, z_a = \gamma = w_{a'}, w_{a'-1}, \dots w_0 = w$ is a $c$-chain of factorizations connecting $z$ and $w$. Thus, $\mathsf{c}(x) \le c$ for every $x \in M$.
	
	Finally, suppose that $z_1$ and $z_2$ are two different factorizations of the same element $x$ such that $\mathsf{d}(z_1, z_2) < c$. Then every atom $a$ appears at most $c-1$ times in $z_1 - \gcd(z_1, z_2)$  as well as in $z_2 - \gcd(z_1, z_2)$. Hence every coefficient of the polynomial $z_1(X) - z_2(X)$ is smaller than~$c$, which contradicts the fact that $m(X)$ divides it. We have then that $\mathsf{c}(x) \ge c$ for every $x \in M$. Thus, we conclude that $\mathsf{c}(M) = c$.
	
	(2) Suppose that $c \in \nn_{\ge 3}$ and $d \in \nn$. Now consider the polynomial $m(X) = c X^d - p$ for some $p \in \pp$ with $p < c$ and $p \nmid c$.   Because $m(0) < 0$ and $m(1) > 0$, the polynomial $m(X)$ has some positive root $\alpha$ such that $\alpha < 1$. We know from part (1) that $m(X)$ is irreducible over $\zz$ and, therefore, it is the minimal polynomial of~$\alpha$. Now set $M := \nn_0[\alpha]$. 
	Using the same arguments as in part (1) we obtain that $\text{rank }M = d$ and $M$ is atomic with $\mathcal{A}(M) = \{\alpha^n \mid n \in \nn_0\}$. On the other hand, it follows from \cite[Theorem~4.7]{CG22} that $M$ does not satisfy the ACCP. 
	
	\smallskip
	\noindent \textit{Claim.} For every $z \in \mathsf{Z}(M)$ there exists $\gamma_z \in \mathsf{Z}(\pi(z))$ such that $s(\gamma_z) < d $ and  we can find a $c$-chain of factorizations connecting $z$ and $\gamma_z$.
	
	\smallskip
	\noindent \textit{Proof of Claim.}  We use induction on $s(z)$. If $s(z) < d$, then we can take $\gamma_z = z$ and our claim follows. Assume that $k \in \nn_0$ and that our claim holds for every $z \in \mathsf{Z}(M)$ such that $s(z) < k$.  Now take $z \in \mathsf{Z}(x)$ such that $s(z) = k$ for some $x \in M$. Write $z = \sum_{i=0}^{t} c_i \alpha^i$ where $c_i \in \nn_0$ for every $i \in \ldb 0,t \rdb$. Observe that $c \alpha^{k} = p \alpha^{k-d}$ since $\alpha$ is a root of $X^{k - d} m(X)$. Let $a$ and $b$ be the unique positive integers satisfying that $c_k = ca + b$ and  $b < c$. Similarly as in part (1) we can take $z = z_0, z_1, \dots, z_a \in \mathsf{Z}(x)$ such that for each $j \in \ldb 1,a \rdb$ the factorization $z_j$ is obtained from $z_{j-1}$ after replacing $c$ occurrences of the atom $\alpha^k$  by $p$ occurrences of $\alpha^{k-d}$. Hence $z_0, z_1, \dots, z_a$ is a $c$-chain of factorizations of $x$. Moreover, $s(z_a) < s(z) = k$, and by induction hypothesis  there exists a $\gamma_z \in \mathsf{Z}(x)$ with $s(\gamma_z) < d $ and some $c$-chain of factorizations connecting $z_a$ and $\gamma_z$. Our claim follows after concatenating both chains. 
	
	Using our claim and mimicking what we did in part (1), we conclude that $\mathsf{c}(M) = c$.
\end{proof}

As the following example illustrates, there are positive algebraic $\alpha$ such that $\nn_0[\alpha]$ is atomic and $\mathsf{c}(\nn_0[\alpha]) = \infty$.

\begin{example}
	Consider the polynomial $m(X) = X^4 - X^3 - X^2 - X + 1 \in \zz[X]$. We have seen in Example~\ref{ex:set of lengths that is not an AAP} that $m(X)$ has an algebraic root $\alpha > 1$. We have seen in the same example that $ \nn_0[\alpha]$ is atomic with $\mathcal{A}( \nn_0[\alpha]) = \{\alpha^n \mid n \in \nn_0 \}$ and also that, for each $k \in \nn$, the element $y_k := \alpha^{3k+1} + 1$ has the following two factorizations:
	\[
		z_1 := \alpha^{3k+1} + 1 \quad \text{ and } \quad z_2 := \alpha^{3k} + \alpha + \sum_{i=1}^k \alpha^{3i-1}.
	\]
        Now let us assume that $\mathsf{c}(\nn_0[\alpha]) = c \in \nn$. We have seen in the discussion of Example~\ref{ex:set of lengths that is not an AAP} that there exists $k \in \nn$ such that $\mathsf{L}(y_k) \cap \ldb 3, c + 2 \rdb$ is empty. Now let's consider a $c$-chain between $z_1$ and $z_2$: $w_0 = z_1, w_1, \ldots, w_\ell = z_2$. Let $i_0 = \min \left\{ i \in \ldb 1, \ell \rdb \mid \lvert w_i \rvert > 2 \right\}$. Then $\lvert w_{i_0} \rvert \in \ldb 3, c + 2 \rdb$, a contradiction with our assumption that $\mathsf{L}(y_k) \cap \ldb 3, c + 2 \rdb = \emptyset$.
 
	As a result, $\mathsf{c}(\nn_0[\alpha]) = \infty$.
\end{example}

\bigskip
\section*{Acknowledgments}

The authors are grateful to their mentor, Dr. Felix Gotti, for proposing this project and for his guidance during the preparation of this paper. While working on this paper, the authors were part of CrowdMath, a massive year-long math research program hosted by MIT PRIMES and the Art of Problem Solving (AoPS). The authors express their gratitude to the directors and organizers of CrowdMath, MIT PRIMES, and AoPS for making this research experience possible.

\bigskip \bigskip

\end{document}